\documentclass[a4paper, 12pt]{amsart}

\usepackage{amssymb}
\usepackage{amsthm}
\usepackage{amsmath}
\usepackage{mathrsfs}
\usepackage{comment}

\theoremstyle{plain}
\newtheorem{thm}{Theorem}[section]		
\newtheorem{prop}[thm]{Proposition}
\newtheorem{cor}[thm]{Corollary}
\newtheorem{lem}[thm]{Lemma}

\theoremstyle{definition}
\newtheorem{df}{Definition}[section]

\theoremstyle{remark}
\newtheorem{rmk}{Remark}[section]
\newtheorem*{ac}{Acknowledgements}

\newcommand{\nn}{\mathbb{N}}
\newcommand{\zz}{\mathbb{Z}}

\newcommand{\rr}{\mathbb{R}}

\newcommand{\uu}{\mathcal{U}}

\DeclareMathOperator{\card}{card}

\DeclareMathOperator{\cl}{CL}

\DeclareMathOperator{\pc}{PC}

\DeclareMathOperator{\kpc}{KPC}
\DeclareMathOperator{\cdim}{Cdim}

\makeatletter
\@addtoreset{equation}{section}

\makeatother

\begin{document}

\title[On  the Assouad dimension and convergence]
{On the Assouad dimension and convergence of metric spaces}
\author[Yoshito Ishiki]
{Yoshito Ishiki}
\address[Yoshito Ishiki]
{\endgraf
Graduate School of Pure and Applied Sciences
\endgraf
University of Tsukuba
\endgraf
Tennodai 1-1-1, Tsukuba, Ibaraki, 305-8571, Japan}
\email{ishiki@math.tsukuba.ac.jp}

\date{January 16, 2020}
\thanks{The author was supported by JSPS KAKENHI Grant Number 18J21300.}

\begin{abstract}
We introduce the notion of pseudo-cones of metric spaces as a generalization of  both of the tangent cones and the asymptotic cones. 
We prove that the Assouad dimension of a metric space is bounded from below by that of any pseudo-cone of it. 
We exhibit a example containing all compact metric spaces as pseudo-cones,
and examples containing all proper length spaces as  tangent cones or asymptotic cones. 
\end{abstract}
\maketitle

\section{Introduction}
Assouad \cite{A1, A2, A3} introduced 
the notion of the so-called Assouad dimension for metric spaces, 
and studied  the relation between the  bi-Lipschitz embeddability
 into a Euclidean space for metric spaces and their Assouad dimensions.
In general, 
it seems to be difficult to estimate the Assouad dimension from below.   
Mackay and Tyson \cite{MT} provided 
a lower estimation of the  Assouad dimensions of metric spaces by using  their tangent spaces. 
Namely, they proved that
if $W$ is a tangent space of a metric space $X$, 
then $\dim_AW\le \dim_AX$, where $\dim_A$ stands for the Assouad dimension. 

Le Donne and Rajala \cite{LR} obtained  both-sides estimations of  the Assuad dimensions and the Nagata dimensions by tangent spaces under a certain assumption (see \cite[Thereoms 1.2, 1.4]{LR}). 
Dydak and Higes \cite{DH} obtained a lower estimation  of the Assouad-Nagata dimensions
by asymptotic cones as ultralimits (see \cite[Proposition 4.1]{DH}). 

In this paper, 
we introduce the notion of pseudo-cones of metric spaces, 
which  can be considered as a generalization of tangent cones and asymptotic cones. 
For $h\in (0, \infty)$, 
and for a metric space $X$ with metric $d_X$,  
we denote by $hX$ the metric space $(X, hd_X)$. 
\begin{df}[Pseudo-cone]
Let $X$ be a metric space. 
Let $\{A_i\}_{i\in \nn}$ be a sequence of subsets of $X$,  
and let  $\{u_{i}\}_{i\in \nn}$ be a sequence in $(0, \infty)$. 
We say that a metric space $P$ is 
a \emph{pseudo-cone of $X$ approximated by $(\{A_i\}_{i\in \nn}, \{u_i\}_{i\in \nn})$} if 
$\lim_{i\to \infty}d_{GH}(u_iA_i, P)=0$, where $d_{GH}$ is the Gromov--Hausdorff distance.
\end{df}

For instance, 
every closed ball centered at a based point of a tangent cone  or an asymptotic cone  is a pseudo-cone. 
Indeed,  
if a pointed metric space $(W,w)$ is a tangent (resp. asymptotic) cone of a metric space $X$ at $p$,
 then  
for every $R\in (0, \infty)$ the closed ball $B(w, R)$ centered at $w$ with radius $R$ is a Gromov--Hausdorff limit of $r_iB(p_i, R/r_i)$, 
where $\{p_i\}_{i\in \nn}$ is a sequence in $X$ with $\lim_{i\to \infty}p_i=p$ and $\lim_{i\to \infty}r_i=\infty$ (resp. $0$);
in particular, 
$B(w, R)$ is a pseudo-cone of $X$ approximated by $(\{B(p_i, R/r_i)\}_{i\in \nn}, \{r_i\}_{i\in \nn})$.

For  a metric space $X$,
 we denote by $\pc(X)$ the class of all pseudo-cones of $X$.
By using the notion of pseudo-cones,  
we formulate a generalization of the Mackay--Tyson  estimation for the Assouad dimensions. 
\begin{thm}\label{thm:cone}
Let $X$ be a metric space. 
Then for every $P\in \pc(X)$ we have
\[
\dim_AP\le \dim_AX.
\] 
\end{thm}

The notion of ultralimits of metric spaces is
 a method of emulating a limit space of a sequence of metric spaces (see Subsection \ref{subsec:ult}). 
Let $\uu$ be a non-principal ultrafilter on $\nn$. 
For a sequence $\{(X_i, p_i)\}_{i\in \nn}$ of pointed metric spaces, 
we denote by $\lim_{\uu}(X_i, p_i)$ the ultralimit of $\{(X_i, p_i)\}_{i\in \nn}$
with respect to $\uu$. 
The existence of an ultralimit is always guaranteed.

For an ultralimit analogy of pseudo-cones, we obtain the following:
\begin{thm}\label{thm:ultcone}
Let $X$ be a metric space. Let $\{A_i\}_{i\in \nn}$ be a sequence of subsets of $X$, 
and let $\{u_i\}_{i\in \nn}$ be a sequence in $(0, \infty)$. 
Take  $a_i\in A_i$ for each $i\in \nn$.  
Then 
for every non-principal ultrafilter $\uu$ on $\nn$ we have 
\[
\dim_A\left(\lim_{\uu}(u_iA_i, a_i)\right)\le \dim_AX. 
\]
\end{thm}

The lower Assouad dimension was essentially introduced by Larman \cite{Lar}. 
This dimension is used for interpolation of the Assouad dimension.
We also obtain the similar estimations as Theorems \ref{thm:cone}  and \ref{thm:ultcone} for the lower Assouad dimension 
(see Theorems \ref{thm:lower} and \ref{thm:ultlower}).

In the conformal dimension theory, 
the conformal Assouad dimension is studied as an invariant of quasi-symmetric maps (see \cite{MT}). 
In other words, 
by comparing the conformal  Assouad dimensions of metric spaces, 
we can distinguish their quasi-symmetric equivalent classes. 
In general,  
it seems to be quite difficult to find the exact value of  the conformal Assouad dimension.

For a metric space $X$, 
we denote by $\kpc(X)$ the class of all pseudo-cones approximated by 
a pair of  a sequence  $\{A_i\}_{i\in \nn}$ of compact sets of  $X$ and a sequence $\{u_i\}_{i\in \nn}$ in $(0, \infty)$. 
We also obtain the following lower estimation of the conformal Assouad dimensions: 
\begin{thm}\label{thm:coneconf}
Let $X$ be a metric space. 
Then for every $P\in \kpc(X)$ we have 
\[
\cdim_AP\le \cdim_AX, 
\]
where $\cdim_A$ stands for the conformal Assouad dimension. 
\end{thm}
As a consequence of Theorem \ref{thm:coneconf}, 
 for every metric space $X$,  
 we can estimate the conformal Assouad dimension of $X$ 
 by the conformal Assouad dimension of closed balls of an ultralimit constructed by scalings of subsets of $X$ (see Corollary \ref{cor:ballconf}).  
 
 The points in the proofs of Theorems \ref{thm:cone}, \ref{thm:ultcone} and \ref{thm:coneconf} for pseudo-cones or ultralimits are to use the stability of the Assouad dimension under scaling of metrics,  and  to 
 extract the techniques of Mackay and Tyson \cite{MT} in their lower estimation for  their tangent spaces. 
 
We say that a topological space $X$ is an \emph{$(\omega_0+1)$-space} if
 $X$ is homeomorphic to the one-point compactification of the countable discrete topological space.
For example, 
the ordinal space $\omega_0+1$ with the order topology is an $(\omega_0+1)$-space. 

We construct an $(\omega_0+1)$-metric space containing any compact metric space as a pseudo-cone. 
\begin{thm}\label{thm:ury}
There exists an $(\omega_0+1)$-metric space $X$ such that 
$\pc(X)$ contains all compact metric spaces. 
\end{thm}

A metric space is said to be  a \emph{length space} if
 the distance of two points in the metric space is equal to the infimum of lengths of arcs jointing the two points. 
A metric space is said to be \emph{proper} if all bounded closed sets in the metric space are compact. 

Similarly to Theorem \ref{thm:ury}, 
we construct metric spaces containing any proper length space as a tangent cone or an asymptotic cone. 

\begin{thm}\label{thm:tancone}
There exists an $(\omega_0+1)$-metric space $X$ for which
 every pointed proper length space $(K, p)$ is a tangent cone of $X$ at its unique accumulation point. 
\end{thm}

\begin{thm}\label{thm:asymcone}
There exists a proper countable discrete metric space $X$ 
for which  every pointed proper length space $(K, p)$ is an asymptotic cone of $X$ at some point. 
\end{thm}

The metric  spaces mentioned in Theorems \ref{thm:ury}, \ref{thm:tancone} or \ref{thm:asymcone} are
 examples showing that analogies of Theorem \ref{thm:cone} for
  the topological dimension, the Hausdorff dimension and the conformal Hausdorff dimension are false (see Proposition \ref{prop:counter}).

The organization of this paper is as follows: 
In Section \ref{sec:pre}, 
we review the definitions and basic properties of the Assouad dimension and the Gromov--Hausdorff distance. 
In Section \ref{sec:cone}, 
we prove Theorems \ref{thm:cone}, \ref{thm:ultcone} and \ref{thm:coneconf}. 
In Section \ref{sec:examples}, 
we construct examples ststed in Theorems \ref{thm:ury}, \ref{thm:tancone} and \ref{thm:asymcone}.

Fraser and Yu \cite{FY} studied a characterization  of subspaces of full Assouad dimension in the Euclidean space from a viewpoint of number theory. 
In \cite{FY}, 
they observed a variant of the Mackay--Tyson estimation in a Euclidean setting. 
In the forthcoming paper \cite{Ishi},
as an application of 
 Theorem \ref{thm:cone}, 
the author will generalized 
 the characterization theorem of Fraser and Yu to a metric space setting. 
In \cite{Ishi},  the author will introduce the notion of  \emph{tiling spaces},  which form a subclass of metric spaces including  the Euclidean spaces, 
the middle-third Cantor spaces, 
and various self-similar spaces appeared in fractal geometry. 
In \cite{Ishi}, the author will prove that for every doubling tiling space $X$, a subset  $F$ of $X$ satisfies $\dim_AF=\dim_AX$ if and only if 
 some specific subset of $X$,  called a \emph{tile},  is  a pseudo-cone of $F$.

\begin{ac}
The author would like to thank Professor Koichi Nagano for his advice and constant encouragement. 
The author also would like to thank Enrico Le Donne,  
Tapio Rajala and Jerzy Dydak for their helpful comments on the references.  
\end{ac}

\section{Preliminaries}\label{sec:pre}
Let $X$ be a metric space. 
The symbol $d_X$  stands for the metric of $X$. 
Let $A$ be a subset of $X$.  
We denote by $\delta(A)$ the diameter of $A$, 
and
we set  $\alpha(A)=\inf\{\, d_X(x, y)\mid \text{$x, y\in A$ and $x\neq y$}\, \}$. 
We denote by $B(x,r)$ (resp. $U(x,r)$) the closed (resp. open) ball centered at $x$ with radius $r$. 
We also denote by $B(A, r)$ the set $\bigcup_{a\in A}B(a,r)$. 
To emphasize a metric space under consideration, 
we often use symbols $\delta_{X}(A)$,  $\alpha_{X}(A)$,  
$B(x, r;X)$ and $B(A, r; X)$ instead of $\delta(A)$, 
$\alpha(A)$, $B(x, r)$ and $B(A, r)$, 
respectively.  
A subset $A$ of $X$ is said to be \emph{$r$-separated} if $\alpha(A)\ge r$.
A subset $A$ is \emph{separated} if it is $r$-separated for some $r$.  

In this paper, we denote by $\nn$ the set of all non-negative integers.

\subsection{Assouad dimension}
For a positive integer $N\in \nn$,
a metric space $X$ is said to be \emph{$N$-doubling} 
if for every bounded set $S\subset X$ 
there exists a subset $F\subset X$ such that 
$S\subset B(F, \delta(S)/2)$ and $\card(F)\le N$. 
Note that if a metric  space $X$ is $N$-doubling, 
then so are  all subsets of $X$. 
A metric space is \emph{doubling} if
 it is $N$-doubling for some $N$.\par

For a bounded set $S\subset X$, 
we denote by $\mathcal{B}_X(S, r)$ the minimum integer $N$ such that 
$S$ can be covered by at most $N$ bounded sets with diameter at most $r$. 
We denote by $\mathscr{A}(X)$ 
the set of all $\beta\in (0, \infty)$
for which
here exists $C\in (0,\infty)$ such that
for every bounded set $S\subset X$, 
and for every $r\in (0, \infty)$,  
we have $\mathcal{B}_X(S, r)\le C(\delta(S)/r)^{\beta}$. 

The \emph{Assouad dimension $\dim_AX$ of $X$} is defined  as $\inf (\mathscr{A}(X))$ if
 the set $\mathscr{A}(X)$ is non-empty; otherwise, $\dim_AX=\infty$. 
We denote by $\mathscr{B}(X)$ the set of  all $\beta\in (0,\infty)$ for which 
there exists $C\in (0, \infty)$ such that for every finite subset $A$ of $X$ 
we have $\card(A)\le C(\delta(A)/\alpha(A))^{\beta}$, where $\card(A)$ stands for the cardinality of $A$. 

By definitions, 
we obtain the next two propositions. 
\begin{prop}\label{prop:Assouaddim}
For every metric space $X$, 
the following are equivalent:
\begin{enumerate}
 \item  $X$ is doubling;
\item $\mathscr{A}(X)$ is non-empty;
\item  $\mathscr{B}(X)$ is non-empty;  
\item  $\dim_AX<\infty$. 
\end{enumerate}
\end{prop}

\begin{prop}\label{prop:Assouad}
For every metric space $X$, 
we have 
\[
\dim_AX=\inf(\mathscr{B}(X)). 
\]
\end{prop}
The \emph{lower Assouad dimension $\dim_{LA}X$ of $X$} is defined as 
the supremum 
of all $\beta\in (0, \infty)$ for which 
there exists 
$C\in (0, \infty)$ such that for every finite set $S$ in $X$
we have $\card(S)\ge C(\delta(S)/\alpha(S))^{\beta}$.

\begin{prop}
For every metric space $X$, we have 
\[
\dim_{LA}X\le \dim_AX. 
\]
\end{prop}

\subsection{Gromov--Hausdorff distance}
For a metric space $Z$,
 and  for subsets $S, T\subset Z$, 
we define the Hausdorff distance $d_H(S,T;Z)$
 between $S$ and $T$ in $Z$ as the infimum $r\in (0, \infty)$ for which
$S\subset B(T, r)$ and $T\subset B(S, r)$. 
For two metric spaces $X$ and $Y$, 
the \emph{Gromov--Hausdorff distance} $d_{GH}(X,Y)$ 
between $X$ and $Y$ is 
defined as the infimum of  all values  $d_H(i(X), j(Y); Z)$, 
where $Z$ is a metric space and
 $i:X\to Z$ and $j:Y\to Z$ are isometric embeddings. 

To deal with the Gromov--Hausdorff distance, 
we use the so-called approximation maps. 
For $\epsilon\in (0, \infty)$, 
and for metric spaces $X$ and $Y$, 
a pair $(f,g)$ with $f:X\to Y$ and $g:Y\to X$ is said to be 
an \emph{$\epsilon$-approximation} 
if the following conditions hold: 
\begin{enumerate}
\item for all $x, y\in X$, we have $|d_X(x,y)-d_Y(f(x),f(y))|<\epsilon$; 
\item for all $x,y\in Y$, we have $|d_Y(x,y)-d_X(g(x), g(y))|<\epsilon$; 
\item for each $x\in X$ and for each $y\in Y$, 
we have $d_X(g\circ f(x),x)<\epsilon$ 
and $d_Y(f\circ g(x), x)<\epsilon$. 
\end{enumerate}

By the definitions of the Gromov--Hausdorff distance, 
we obtain:
\begin{prop}\label{prop:distance}
Let $X$ and $Y$ be metric spaces. 
Then for every $h\in (0, \infty)$ we have $d_{GH}(hX, hY)=hd_{GH}(X,Y)$. 
\end{prop}

The next two claims can be seen in  Sections 7.3 and 7.4 in \cite{BBI}. 
\begin{lem}\label{lem:approx}
If  metric spaces $X$ and $Y$ satisfy $d_{GH}(X, Y)\le \epsilon$, 
then there exists a $2\epsilon$-approximation between them.  
\end{lem}
\begin{prop}\label{prop:unidiameter}
For all metric spaces $X$ and $Y$, we have 
\[
|\delta(X)-\delta(Y)|\le 2d_{GH}(X,Y). 
\]
\end{prop}

We say that a sequence $\{(X_i, p_i)\}_{i\in \nn}$ of pointed metric spaces
 \emph{converges to  $(Y, q)$ in the pointed Gromov--Hausdorff topology } if 
there exist a sequence  $\{\epsilon_{i}\}_{i\in \nn}$ in $(0, \infty)$ with 
$\lim_{i\to \infty}\epsilon_{i}=0$, and a sequence $\{(f_i, g_i)\}_{i\in \nn}$ of
 $\epsilon_{i}$-approximation maps between
  $X_i$ and $Y$ with $f_i(p_i)=q$ and $g(q)=p_i$. 

\subsection{Ultralimits}\label{subsec:ult}
Let $\mathcal{U}$ be  a non-principal ultrafilter on $\nn$. 
For a sequence  $\{a_i\}_{i\in \nn}$ in $\rr$, 
a real number $u$ is a \emph{ultralimit} of $\{a_i\}_{i\in \nn}$ with respect to $\uu$ if
for every $\epsilon\in (0, \infty)$ 
we have $\{\, i\in \nn\mid |a_i-u|<\epsilon\, \}\in \uu$.  
In this case, we write $\lim_{\uu}a_i=u$. 
Note that an ultralimit of a bounded sequence in $\rr$ always uniquely exists. 

Let $\{(X_i, p_i)\}_{i\in \nn}$ be a sequence of pointed metric spaces. 
We put 
\[
B(\{(X_i, p_i)\}_{i\in \nn})=\left\{\,\{x_i\}_{i\in \nn}\in \prod_{i\in \nn}X_i\mid \sup_{i\in \nn}d(p_i, x_i)<\infty\, \right\}.
\] 
Define an equivalence relation $R_{\uu}$ on $B(\{(X_i, p_i)\}_{i\in \nn})$ in such a way that 
$\{x_i\}_{i\in\nn}R_{\uu}\{y_i\}_{i\in \nn}$ if and only if
 $\lim_{\uu}d_{X_i}(x_i, y_i)=0$. 
We denote by $[\{x_i\}_{i\in \nn}]$ 
the equivalence class of $\{x_i\}_{i\in \nn}$. 
We denote by  $\lim_{\uu}(X_i, p_i)$ 
the metric space $B(\{(X_i, p_i)\}_{i\in \nn})/R_{\uu}$ equipped with 
the metric $d_{\lim_{\uu}(X_i, p_i)}$ defined by 
\[
d_{\lim_{\uu}(X_i,p_i)}(x, y)=\lim_{\uu}d_{X_i}(x_i, y_i), 
\] 
where $x=[\{x_i\}_{i\in \nn}]$ and $y=[\{y_i\}_{i\in \nn}]$. 
We call $\lim_{\uu}(X_i,p_p)$ 
the \emph{ultralimit} of the sequence $\{(X_i, p_i)\}_{i\in \nn}$ with respect to $\uu$. 

The following can be seen   in \cite[I.5.52]{BH} or \cite[Proposition 3.2]{KL}. 
\begin{lem}\label{lem:ultiso}
Let $\{(X_i, p_i)\}_{i\in \nn}$ be a sequence of pointed compact metric spaces. 
If $\{(X_i, p_i)\}_{i\in \nn}$ converges to a pointed compact metric space $(X, p)$, 
then $\lim_{\uu}(X_i, p_i)$ and $(X, p)$ are isometric to each other. 
\end{lem}

We say that  a sequence $\{X_i\}_{i\in \nn}$ of metric spaces is 
\emph{uniformly bounded} if there exists $M\in (0, \infty)$ 
with  $\sup_{i\in \nn}\delta(X_i)\le M$. 

For every uniformly bounded sequence $\{X_i\}_{i\in \nn}$, 
and for every choice $\{p_i\}_{i\in \nn}\in \prod_{i\in \nn}X_i$ of base points, 
we have 
$B(\{(X_i, p_i)\}_{i\in \nn})=\prod_{i\in\nn}X_i$. 
 Therefore Lemma \ref{lem:ultiso} implies:
\begin{lem}\label{lem:inyoulemma}
Let $\{X_i\}_{i\in \nn}$ be a uniformly bounded sequence of compact metric spaces. 
If  $\{X_i\}_{i\in \nn}$ converges to a compact metric space $L$,  
then for every choice $\{p_i\}_{i\in \nn}\in \prod_{i\in \nn}X_i$ of base points, 
the metric spaces $\lim_{\uu}(X_i, p_i)$ and $L$ are isometric to each other. 
\end{lem}

\section{Pseudo-cones and the Assouad dimension}\label{sec:cone}
In this section, 
we prove Theorems \ref{thm:cone}, \ref{thm:ultcone} and \ref{thm:coneconf}. 

\subsection{Basic properties of pseudo-cones}

By Proposition \ref{prop:distance}, 
we have:
\begin{prop}
Let $X$ be a metric space. 
If $A\in \pc(X)$, 
then for every $h\in (0, \infty)$ we have $hA\in \pc(X)$. 
\end{prop}

We say that 
a sequence $\{X_i\}_{i\in \nn}$ of metric space  is
\emph{uniformly precompact} if 
for every $\epsilon\in (0, \infty)$ there exists $M\in \nn$ such that for 
every $i\in \nn$,  every  $\epsilon$-separated set in $X_i$ has 
at most $M$ elements. 
Note that if there exists $N\in \nn$ such that for each $i\in \nn$, 
the space $X_i$ is $N$-doubling, 
then $\{X_i\}_{i\in \nn}$ is uniformly bounded. 

We recall Gromov's precompactness theorem
 (see Section 7.4 in \cite{BBI}).  
Namely, 
if a sequence $\{X_i\}_{i\in \nn}$ of compact metric spaces is 
uniformly precompact and uniformly bounded, 
then there exists a subsequence of $\{X_i\}_{i\in \nn}$ 
which converges to a compact metric space in the sense of Gromov--Hausdorff. 
This guarantees the existence of pseudo-cones for doubling metric spaces.  

\begin{prop}
Let $X$ be a doubling metric space.   
 If a sequence $\{A_i\}_{i\in \nn}$ consists of compact sets in $X$, 
 and if  $\{u_iA_i\}_{i\in \nn}$ is uniformly bounded  
 for a sequence  $\{u_i\}_{i\in \nn}$ in $(0, \infty)$, 
 then there exists a convergent subsequence
  $\{u_{\phi(i)}A_{\phi(i)}\}_{i\in\nn}$ of $\{u_iA_i\}_{i\in \nn}$
   in the sense of  Gromov--Hausdorff. 
\end{prop}

Let $X$ be a proper metric space, and $p\in X$. 
A pointed metric space $(Y, y)$ is said to be a
 \emph{tangent} (resp. \emph{asymptotic}) \emph{cone of $X$ at $p$} 
if there exist a sequence 
 $\{p_i\}_{i\in \nn}$ in $X$ with $\lim_{i\to \infty}p_i=p$,
  and a sequence  $\{r_i\}$ in $(0, \infty)$ with
   $\lim_{i\to \infty}r_i= 0$ (resp. $\infty$) 
such that for every $R\in (0, \infty)$
 we have 
 $(r_iB(p_i, R/r_i), p_i)\to (B(y, R), y)$ 
 as 
 $i\to \infty$ in the pointed Gromov--Hausdorff topology 
 (see Section 8.1 in \cite{BBI}). 

By the definitions of the tangent cones and the asymptotic cones, 
we obtain: 
\begin{prop}\label{prop:coneandp}
Let $X$ be a proper metric space, 
and let $(Y, y)$ be a tangent cone of $X$ or asymptotic cone of $X$. 
Then for every $R\in (0, \infty)$ we have $B(y, R)\in \pc(X)$. 
\end{prop}

\subsection{Lower estimations}
First we prove Theorem \ref{thm:cone}. 
\begin{proof}[Proof of Theorem \ref{thm:cone}]
Let $X$ be a metic space, 
and $P\in \pc(X)$. 
We assume that
  $P$
   is approximated by 
   $(\{A_i\}_{i\in \nn}, \{u_i\}_{i\in \nn})$. 
Suppose that $\dim_A(X)<\dim_A(P)$. 
Take $\beta\in \mathscr{B}(X)$ with 
$\dim_A(X)<\beta<\dim_A(P)$. 
Since $\beta \in \mathscr{B}$(X), 
there exists $M\in (0, \infty)$ such that 
for every finite set $T$ of  $X$ we have 
$\card(T)\le M(\delta_X(T)/\alpha_X(T))^{\beta}$. 

 Put $C=4^{\beta}(M+1)$.  
From $\beta<\dim_AP$, 
it follws that $\beta \not\in \mathscr{B}(P)$. 
 Thus there exists a finite set $S$ of 
 $P$ with 
$\card(S)> C(\delta_P(S)/\alpha_P(S))^{\beta}$. 
Since  $d_{GH}(u_iA_i, P)\to 0$ as $i\to \infty$, 
we can take $N\in \nn$ such that
  $d_{GH}(u_NA_N, P)<\alpha_P(S)/20$. 
By Lemma \ref{lem:approx}, 
there exists an $(\alpha_P(S)/10)$-approximation $(f, g)$ between
 $u_NA_N$ and $P$. 
For each $x\in S$, 
take $t_x\in u_NA_N$ such that
  $t_x\in B(g(x), \alpha_P(S)/10)$. 
Note that  if $x\neq y$, then $t_x\neq t_y$. 
Put $T=\{\, t_x\mid x\in S\, \}$. 
For all $x,y\in S$, we obtain
\[
u_N d_X(t_{x}, t_{y})\le d_P(x, y)+3\alpha_P(S)/10 \le 2\delta_P(S), 
\]
and 
\[
u_Nd_X(t_x, t_y)\ge d_P(x, y)-3\alpha_P(S)/10\ge 2^{-1}\alpha_P(S).
\]
Thus,  
we have 
$\delta_X(T)\le 2u_N^{-1}\delta_P(S)$ and $\alpha_X(T)\ge 2^{-1}u_N^{-1}\alpha_P(S)$, 
and hence 
\begin{align*}
\card(T) &=\card(S)> C(\delta_P(S)/\alpha_P(S))^{\beta} \\
&=4^{-\beta}C(2u_N^{-1}\delta_P(S)/2^{-1}u_N^{-1}\alpha_P(S))^{\beta} \ge 4^{-\beta}C(\delta_X(T)/\alpha_X(T))^{\beta}.
\end{align*}
On the other hand,  we also have 
$\card(T)\le M(\delta_X(T)/\alpha_X(T))^{\beta}$.
These inequalities imply that $4^{-\beta}C<M$. This is a contradiction. 
\end{proof}
Since for every metric space $X$ we have $X\in \pc(X)$, by Theorem \ref{thm:cone}, we obtain:
\begin{cor}
Let $X$ and $Y$ be metric spaces. If $d_{GH}(X,Y)=0$, 
then $\dim_AX=\dim_AY$. 
\end{cor}
This corollary slightly generalizes  the fact that the Assouad dimension of any metric space  is equal to  that of its completion. 

By a similar proof to Theorem \ref{thm:cone}, we obtain:
\begin{thm}\label{thm:lower}
Let $X$ be a metric space. 
Then for every $P\in \pc(X)$ we have
\[
\dim_{LA}X\le \dim_{LA}P.
\] 
\end{thm}

Next we prove Theorem \ref{thm:ultcone}
\begin{proof}[Proof of Theorem \ref{thm:ultcone}]
Let $X$ be a metric space. 
Let $\{A_i\}_{i\in \nn}$ be a sequence of subsets of $X$, 
and let $\{u_i\}_{i\in \nn}$ be a sequence  in $(0, \infty)$. 
Let $\uu$ be a non-principal ultrafilter on $\nn$. 
Take $a_i\in A_i$ for each $i\in \nn$. 
Put $P=\lim_{\uu}(u_iA_i, a_i)$. 

Suppose that $\dim_AX<\dim_AP$. 
Let $\beta\in \mathscr{B}(X)$,  
$M\in (0, \infty)$, 
$C=4^{\beta}(M+1)$ 
and 
$S\subset P$ be the same objects in the proof of Theorem \ref{thm:cone}. 
Put $S=\{[x_{1,i}], [x_{1,i}]\dots, [x_{n,i}]\}$. 
Put $S_i=\{x_{1,i}, \dots, x_{n, i}\}\subset u_iA_i$ 
for each $i\in \nn$. 
By the definition of ultralimits, 
for 
$\uu$-almost all $i\in\nn$,  
and for all $k, l\in\{1, \dots, n\}$
we have 
\[
|u_id_X(x_{k, i}, x_{l, i})-d_P([x_{k, i}], [x_{l, i}])|<\alpha(S)/2. 
\]
Then for such $\uu$-almost all $i\in \nn$ we have 
\[
\delta_X(S_i)\le 2u_i^{-1}\delta(S)
\]
and 
\[
\alpha_X(S_i)\ge 2^{-1}u_i^{-1}. 
\]
Since $\card(S_i)=\card(S)$, 
by a similar argument to the proof of Theorem \ref{thm:cone},  
we obtain $4^{-\beta}C<M$. 
This is  a contradiction. 
\end{proof}

By a similar argument to the proof of Theorem \ref{thm:ultcone}, 
we obtain:
\begin{thm}\label{thm:ultlower}
Let $X$ be a metric space. 
Let $\{A_i\}_{i\in \nn}$ be a sequence of subsets of $X$, 
and let $\{u_i\}_{i\in \nn}$ be a sequence in $(0, \infty)$. 
Take $a_i\in A_i$ for each $i\in \nn$. 
Then for every  non-principal ultrafilter $\uu$ on $\nn$  
we have
\[
\dim_{LA}X\le \dim_{LA}\left(\lim_{\uu}(u_iA_i, a_i)\right). 
\]
\end{thm}

\subsection{Conformal Assouad dimension}
Let  $\eta:[0, \infty)\to [0, \infty)$  be a homeomorphism. 
A homeomorphism $f:X\to Y$ between metric spaces is said to be an \emph{$\eta$-quasi-symmetric} map 
if the following holds:
\begin{itemize}
\item[(QS)]
if for $x, y, z\in X$ and for $t\in [0, \infty)$
 we have $d_X(x,y)\le td_X(x, z)$, 
then $d_Y(f(x, f(y)))\le \eta(t)d_Y(f(x), f(z))$. 
\end{itemize}
A homeomorphism $f:X\to Y$ is  \emph{quasi-symmetric} if it is $\eta$-quasi-symmetric for some $\eta$. 
Note that the inverse of a quasi-symmetric map is also quasi-symmetric. 

For a metric space $X$, 
the \emph{conformal Assouad dimension} 
$\cdim_AX$ of $X$ is defined as the infimum of all Assouad dimensions of all quasi-symmetric images of $X$. 

In the proof of Theorem \ref{thm:coneconf}, 
we use the following theorem 
due to Tukia and V\"ais\"al\"a (see \cite[Theorem 2.21]{TV}). 
\begin{thm}\label{thm:tv}
If a map $f:X\to Y$ between metric spaces satisfies the condition \emph{(QS)},  
then $f$ is either a constant map or a quasi-symmetric embedding.
\end{thm}
We now prove Theorem \ref{thm:coneconf}.
\begin{proof}[Proof of Theorem \ref{thm:coneconf}]
Let $X$ be a metric space, and $P\in \kpc(X)$. 
Since a non-doubling space has infinite conformal Assouad dimension, 
we may assume that $X$ is doubling. 
Take a metric space $Y$
 and an $\eta$-quasi-symmetric map $f:X\to Y$. 
We may assume that $P$ is compact and $P$ has at least two elements. 
We assume that $P$ is approximated by 
$(\{A_i\}_{i\in \nn}, \{u_i\}_{i\in \nn})$, 
where $\{A_i\}_{i\in \nn}$ is a sequence of compact sets in $X$. 
By Proposition \ref{prop:unidiameter},  
we have $\sup_{i}\delta(u_iA_i)<\infty$. 
For each $i\in \nn$, put $B_i=f(A_i)$ and $v_i=(\delta_Y(B_i))^{-1}$. 
By Gromov's precompactness theorem, 
by choosing a suitable subsequence if necessary, 
we find 
 a limit compact metric space 
 $Q\in \kpc(Y)$ of $\{v_iB_i\}_{i\in \nn}$. 

Let $\uu$ be a non-principal ultrafilter on $\nn$. 
By Lemma \ref{lem:inyoulemma}, 
we can consider that
 $Q=\lim_{\uu}v_iB_i$ and $P=\lim_{\uu}u_iA_i$. 
Since $f$ is continuous and  $\delta(v_iB_i)=1$ for all $i\in \nn$, 
the map $f:X\to Y$ induce a map 
$F:P\to Q$ defined by 
$F([\{x_i\}_{i\in \nn}])=[\{f(x_i)\}_{i\in \nn}]$. 
Replacing the role of $f$ with that of $f^{-1}$, 
we obtain the inverse of $F$. 
Thus $F$ is bijective. 

In order to prove that $F$ satisfies the condition (QS), 
we assume  $d_P(x, y)\le td_P(x, z)$, 
where $x=[\{x_i\}_{i\in \nn}]$, 
$y=[\{y_i\}_{i\in \nn}]$, 
 $z=[\{z_i\}_{i\in\nn}]$. 
For each $\epsilon\in (0, \infty)$,  
we have 
\[
d_{u_iA_i}(x_i, y_i)<(t+\epsilon)d_{u_iA_i}(x_i, z_i)
\]
for $\uu$-almost all $i\in \nn$. 
Thus, 
since $f$ is $\eta$-quasi-symmetric, 
we have 
\[
d_{v_iB_i}(f(x_i),f( y_i))<\eta(t+\epsilon)d_{v_iB_i}(f(x_i), f(z_i))
\]
for $\uu$-almost all $i\in \nn$.
Then we conclude 
\[
d_Q(F(x), F(z))<\eta(t+\epsilon)d_Q(F(x),F(z)).
\] 
Letting $\epsilon\to 0$, 
we obtain $d_Q(F(x), F(z))\le\eta(t)d_Q(F(x),F(z))$.
Since $F$ is bijective and  non-constant, 
by Theorem \ref{thm:tv} we conclude that $F$ is an $\eta$-quasi-symmetric map. 
Thus $\cdim_AP\le \dim_AQ$. 
Theorem \ref{thm:cone} implies 
$\dim_AQ\le \dim_AY$. 
In particular, 
$\cdim_AP\le \cdim_AX$. 
\end{proof}

As a corollary of Theorem \ref{thm:coneconf},  
we obtain: 
\begin{cor}\label{cor:ballconf}
Let $X$ be a metric space. 
Let $\{A_i\}_{i\in \nn}$ be a sequence of subsets of $X$, 
and let $\{u_i\}_{i\in \nn}$ be a sequence in $(0, \infty)$. 
Take $p_i\in A_i$ for each $i\in \nn$. 
Put $Y=\lim_{\uu}(u_iA_i, p_i)$. 
Then for every $R\in (0, \infty)$ 
we have 
\[
\cdim_A(B(p, R; Y))\le \cdim_AX, 
\]
where $p=[\{p_i\}_{i\in \nn}]$.
\end{cor}
\begin{proof}
Since  the Assouad of the completion of $X$ coincides with that of $X$, 
we may assume that $X$ is doubling and complete. 
Since $\lim_{\uu}(u_iA_i)$ is isometric to $\lim_{\uu}(u_i\cl(A_i))$, 
we may assume that each $A_i$ is closed. 
Note that $A_i$ is doubling and complete, and hence it is proper. 
Put 
\[
S=\left\{\{x_i\}_{i\in \nn}\in \prod_{i\in \nn}A_i\mid d_{A_i}(p_i, x_i)<2R \right\}/R_{\uu}, 
\]
and 
\[
T=\left\{\{x_i\}_{i\in \nn}\in \prod_{i\in \nn}A_i\mid d_{A_i}(p_i, x_i)\le2R \right\}/R_{\uu}. 
\]
By the definition of ultralimit, 
we have $B(p, R; \lim_{\uu}(u_iA_i))\subset S$. 
We also have 
$T=\lim_{\uu}(B(p_i, 2R), p_i)$. 
Since $A_i$ is proper, 
the ball $B(p_i, R)$ is compact. 
Thus by Theorem \ref{thm:coneconf}, 
we obtain
\[
\cdim_AT\le \cdim_AX. 
\]
By the monotonicity of the conformal Assouad dimension, 
we obtain the corollary. 
\end{proof}
\section{Examples}\label{sec:examples}
In this section, 
we study examples containing a large class of metric space 
as their pseudo-cones, tangent cones or asymptotic cones. 
\subsection{Telescope construction}
Let $\mathcal{X}=\{X_i\}_{i\in \nn}$ be a sequence of bounded metric spaces. 
Assume that $\delta(X_i)\le 2^{-i}$. 
Put
\[
T(\mathcal{X})=\{\infty\}\sqcup \coprod_{i\in \nn}X_i, 
\]
and define a metric $d_{T(\mathcal{X})}$ on $T(\mathcal{X})$ by 
\[
d_{T(\mathcal{X})}(x, y)=
\begin{cases}
		d_{X_i}(x,y) & \text{if $x,y\in X_i$ for some $i$,}\\
		\max\{2^{-i}, 2^{-j}\} & \text{if $x\in X_i,y\in X_j$ for some $i\neq j$, }\\
		2^{-i} & \text{if $x=\infty, y\in X_i$ for some $i$,}\\
		2^{-i} & \text{if $x\in X_i, y=\infty$ for some $i$.}
	\end{cases}
\]
We call the metric space $T(\mathcal{X})$  the \emph{telescope space of $\mathcal{X}$}. 
This construction is a specific version of the telescope spaces discussed in \cite{I}. 

\subsection{Proof of Theorem \ref{thm:ury}}

We construct an $(\omega_0+1)$-metric space containing all compact metric spaces as its pseudo-cones. 

We denote by $\mathfrak{S}$ the class of all separable metic spaces. 
We say that a metric space $X$ is \emph{$\mathfrak{S}$-universal}
if every metric space $A\in \mathfrak{S}$ is isometrically embeddable 
into $X$. 
 
 The Urysohn universal space $\mathbb{U}$ (see \cite{U, H2,  Gro}), 
and 
the space $C([0,1])$ of all real valued continuous functions on $[0,1]$
 equipped with the supremum metric (see \cite{B, H2}) are separable 
and $\mathfrak{S}$-universal. 
Note that the space $\ell^{\infty}$ of all bounded sequences equipped with the supremum metic is $\mathfrak{S}$-universal, 
which is known as Fr\'echet's embedding theorem (see \cite{Fre});  
however, 
$\ell^{\infty}$ is not separable. 

By the virtue of the telescope construction, 
and  by the existence of separable $\mathfrak{S}$-universal metric space, we can prove Theorem \ref{thm:ury}. 
\begin{proof}[Proof of Theorem \ref{thm:ury}]
Let $U$ be a separable $\mathfrak{S}$-universal metric space. 
Let $Q$ be a countable dense set  of $U$,
 and let $J=\{K_i\}_{i\in \nn}$ be the set of all finite set of $Q$. 
Put 
$X=T(J)$. 
The metric space $X$ is an $(\omega_0+1)$-metric space, 
and the point $\infty$ is its unique accumulation point. 
Let $K$ be any compact metric space. 
Since $K$ is isometrically embeddable into $U$, 
there exists a subsequence $\{K_{\phi(i)}\}_{i\in \nn}$ of $J$  
with $d_H(K_{\phi(i)}, K; U)\to 0$ as $i\to 0$. 
For each $i\in \nn$ 
we have $(2^{i}\delta(K_{\phi(i)}))^{-1}K_{\phi(i)}\subset X$. 
Thus  $K\in \pc(X)$.  This finishes the proof of Theorem \ref{thm:ury}. 
\end{proof}
By a similar argument, 
we also obtain:
\begin{prop}\label{prop:ury2}
Let $U$ be a separable $\mathfrak{S}$-universal metric space. 
If $Q$ is a countable dense set  of $U$, 
then $\pc(Q)=\mathfrak{S}$. 
\end{prop}

\subsection{Proofs of Theorems \ref{thm:tancone} 
and \ref{thm:asymcone}}

By an argument on arcs in a length space, 
we obtain the following estimation of the Hausdorff distance between concentric  balls. 
\begin{prop}\label{prop:length}
Let $X$ be a length space and  $p\in X$. 
Then for all $r, R\in (0, \infty)$  
we have 
\[
d_{H}(B(p, r), B(p, R); X)\le |r-R|. 
\]
\end{prop}
The following proposition  is 
a key of our construction of the desired spaces in 
 Theorems \ref{thm:tancone} and \ref{thm:asymcone}. 
\begin{prop}\label{prop:conv}
Let $U$ be a metric space, 
and $Q$ a countable dense subset of $U$. 
Let $K$ be a length metric subspace of $U$, 
and $p\in K$. 
For all $i, k \in \nn$, 
put $l_{k, i}=k\cdot 2^{-i}$. 
Assume that a sequence $\{A_i\}_{i\in \nn}$ of subsets of $Q$ 
satisfies the following for every $i\in \nn$:
\begin{enumerate}
\item[(A1)] $p\in A_i$;
\item[(A2)] for each $k\in \{0, \dots, 2^{2i}\}$ we have 
\[
d_{H}(B(p, l_{k, i}; K), B(p, l_{k, i}; A_i); U)\le 2^{-i}. 
\]
\end{enumerate}
Then  for every $R\in (0, \infty)$, 
the sequence $\{(B(p, R; A_i), p)\}_{i\in \nn}$ converges to $(B(p, R; K), p)$ in the pointed Gromov--Hausdorff topology. 
\end{prop}
\begin{proof}
Take $N\in \nn$ with $R<2^{2N}$.  
Then for each $i\ge N$, 
we can take $k\in \{0, \dots, 2^{2i}\}$ with 
\[
l_{k,i}\le  R< l_{k+1, i}. 
\]
By the condition (A2), 
for $m\in \{k, k+1\}$, 
\[
d_H(B(p, l_{m, i}; A_i), B(p, l_{m, i}; K); U)\le 2^{-i}. 
\]
Thus, 
we have 
\begin{equation}\label{eq:11}
B(p, l_{k,i};K)\subset B(B(p, l_{k, i}; A_i), 2^{-i}; U), 
\end{equation}
and 
\begin{equation}\label{eq:12}
B(p, l_{k+1,i};A_i)\subset B(B(p, l_{k+1, i}; K), 2^{-i}; U). 
\end{equation}

Since for $m\in \{k, k+1\}$ 
we have $|R-l_{m, i}|\le 2^{-i}$, 
by Proposition \ref{prop:length}, 
for $m\in \{k, k+1\}$ we have 
\[
d_{H}(B(p,R; K), B(p, l_{m, i}; K); U)\le 2^{-i}. 
\]
Thus we have 
\begin{equation}\label{eq:21}
B(p, R; K)\subset B(B(p, l_{k, i}; K), 2^{-i}; U), 
\end{equation}
and 
\begin{equation}\label{eq:22}
B(p, l_{k+1, i}; K)\subset B(B(p, R; K), 2^{-i}; U). 
\end{equation}
Since $B(p, l_{k, i}; A_i)\subset B(p, R; A_i)$, 
by (\ref{eq:11}) and (\ref{eq:21}), 
we obtain 
\begin{equation}\label{eq:31}
B(p, R; K)\subset B(B(p, R; A_i); 2^{-i+1}; U). 
\end{equation}
Since $B(p, R; A_i)\subset B(p, l_{k+1, i}; A_i)$, 
by (\ref{eq:12}) and (\ref{eq:22}), 
we obtain 
\begin{equation}\label{eq:32}
B(p, R; A_i)\subset B(B(p, R; K); 2^{-i+1}; U). 
\end{equation}
Then, 
by (\ref{eq:31}) and (\ref{eq:32}), 
we have 
\[
d_{H}(B(p, R; K), B(p, R; A_i); U)\le 2^{-i+1}.
\] 
Hence we conclude that  the sequence $\{(B(p, R; A_i), p)\}_{i\in \nn}$ converges to 
$(B(p, R; K), p)$ in the pointed Gromov--Hausdorff topology. 
\end{proof}

A metric space $X$ is said to be \emph{homogeneous} if for all $x, y\in X$, 
there exists an isometry $f:X\to X$ such that $f(x)=y$. 
The spaces $\mathbb{U}$ and  $C([0,1])$ are homogeneous. 

By the definition of homogeneity, 
we obtain: 
\begin{prop}\label{prop:homuni}
Let $U$ be a homogeneous  $\mathfrak{S}$-universal metric space.  
Then for every $q\in U$, and  for every pointed separable metric space $(X, x)$, 
there exists an isometry $f:X\to U$ such that $f(x)=q$. 
\end{prop}

We now prove Theorem \ref{thm:tancone}.
\begin{proof}[Proof of Theorem \ref{thm:tancone}]
We may assume that $K$ has at least two elements. 
Let $U$ be 
a separable homogeneous  $\mathfrak{S}$-universal metric space. 
For instance,  
we can choose $C([0, 1])$ or $\mathbb{U}$ as $U$. 
Let $Q$ be a countable dense subset of $U$. 
Put $I=\{F_{i}\}_{i\in \nn}$ be
 a sequence consisting of all finite subset of $Q$.
 We impose the condition that
  for every finite subset $A$ of $Q$
   there exists infinite many $n\in \nn$ with $F_n=A$. 
By Proposition \ref{prop:homuni}, 
 we may assume that $K\subset U$ and $p\in Q$. 

For each $i\in \nn$, 
set $r_i=(i+1)!\cdot\delta(F_i)$. 
Put $J=\{(r_i)^{-1}F_i\}_{i\in \nn}$. 
Let  
$X=T(J)$. 
The space $X$ is an $(\omega_0+1)$-metric space, 
and $\infty$ is its unique accumulation point. 

Since $K$ is proper,  
we can take a sequence $\{A_i\}_{i\in \nn}$ of finite subsets of $Q$  
satisfying the conditions (A1) and (A2) in Proposition \ref{prop:conv}. 
By the definition of  $I=\{F_i\}_{i\in \nn}$, 
there exists a strictly increasing map $\phi:\nn\to \nn$ such that 
$r_{\phi(i)}F_{\phi(i)}$ is isometric to $A_i$ for each $i\in \nn$. 
Let $q_i\in F_{\phi(i)}$ be a corresponding point to $p\in A_i$. 
Note that $r_{\phi(i)}\to \infty$ as $i\to \infty$. 

To prove that $(K, p)$ is a tangent cone of $X$,  
we show that for each $R\in (0, \infty)$, 
the sequence 
$\{(r_iB(q_i, R/r_i;X), q_i)\}_{i\in \nn}$ converges to $(B(p, R; K), p)$ 
in the pointed Gromov--Hausdorff topology. 
By the definition of $\{r_i\}_{i\in \nn}$, we can take $N\in \nn$ such that 
if $i>N$, 
then we have $R<r_{\phi(i)}\cdot 2^{-\phi(i)+1}$. 
Therefore, 
by the definition of $X$, 
for every $i>N$, 
  the pointed  metric  space $(r_{\phi(i)}B(q_i, R/r_{\phi(i)}; X), q_i)$ is isometric to
$(B(p, R; F_{\phi(i)}), p)$. 
By Proposition \ref{prop:conv}, 
$\{(r_{\phi(i)}B(q_i, R/r_{\phi(i)}; X), q_i)\}_{i\in \nn}$ 
converges to $(B(p, R; K), p)$ 
in the pointed Gromov--Hausdorff topology. 
Since $q_i\to \infty$ in $X$ as $i\to \infty$, 
we conclude that  $(K,p)$ is a tangent cone of $X$ at $\infty$. 
This completes the proof of Theorem \ref{thm:tancone}. 
\end{proof}

We next prove Theorem \ref{thm:asymcone}. 
As a core part  to construct a metric space mentioned in Theorem \ref{thm:asymcone}, 
we begin with the following elementary lemma 
on a surjective map between countable sets, 
which guarantees a polite way of indexing a countable set. 
\begin{lem}\label{lem:surj}
There exists a surjective map $C: \nn\to \nn^2\times \zz$ 
satisfying the following: 
\begin{enumerate}
\item[(B1)] $C(0)=(0,0,0)$\label{item:00}
\item[(B2)] for every $n\in \nn$, 
two points $C(n)$ and $C(n+1)$ are adjunct in $\nn^2\times \zz$; 
namely, 
for every $n\in \nn$ and for every $i\in \{1,2,3\}$, 
we have 
\[
|\pi_i(C(n))-\pi_i(C(n+1))|\le 1;
\]
where $\pi_i$ is the $i$-th projection, \label{item:adjunct}
\item[(B3)] for each $(x, y, z)\in \nn^2\times \zz$, 
there exist infinite many $n\in \nn$ such that $C(n)=(x, y, z)$. \label{item:infinite}
\end{enumerate}
\end{lem}
\begin{proof}
\ref{item:00}
Take a surjective map $A:\nn\to \nn^2\times \zz$ 
satisfying the conditions (B1) and (B2). 
Assume that $A$ does not satisfy the condition (B3). 
Take a surjective map $H:\nn\to \nn$ satisfying the following:
\begin{enumerate}
\item $H(0)=0$;
\item for every $n\in \nn$, the set  $H^{-1}(n)$ is infinite; 
\item for every $n\in \nn$, we have  $|H(n)-H(n+1)|\le 1$. 
\end{enumerate}
For example, 
if for each $n\in\nn$ we put   
$H(n)=\min_{k\in \nn}|n-k^2|$,  
then 
the map $H:\nn\to \nn$ satisfies the conditions mentioned above. 
Put $C=A\circ H$. Then $C$ satisfies the conditions (B1), (B2) and (B3). 
\end{proof}
By the conditions (B1) and  (B2), 
we inductively obtain:
\begin{lem}\label{lem:tukau}
If a surjective map $C:\nn\to \nn^2\times \zz$ 
satisfies the conditions \emph{(B1)} and \emph{(B2)}, 
then for every $n\in \nn$, and for every $i\in \{1,2,3\}$,  
we have  
\begin{align*}
|\pi_{i}(C(n))|\le n. 
\end{align*}
\end{lem}
We now show that 
the existence of a metric space containing all proper length space 
as its asymptotic cones. 
Such a space is constructed as follows: 
Let $U$ be a separable homogeneous 
$\mathfrak{S}$-universal metric space, 
and let $Q$ be a countable dense subset of $X$. 
For each $(j, k)\in \nn\times \zz$, 
let $I_{(j, k)}=\{F_{(i, j, k)}\}_{i\in \nn}$ be 
a sequence consisting of all finite subsets of $Q$ 
satisfying the following for every $i\in \nn$:
\begin{enumerate}
\item[(C1)]  $q\in F_{(i, j, k)}$; 
\item[(C2)] $2^{-k}\le \delta(F_{(i, j, k)})<2^{-k+1}$;
\item[(C3)] $2^{-j}\le \alpha(F_{(i, j, k)})/\delta(F_{(i, j, k)})< 2^{-j+1}$. 
\end{enumerate}

Take a surjective map $C:\nn\to \nn^2\times \zz$ stated in Lemma \ref{lem:surj}. 
For each $i\in \nn$, 
define $G_i=F_{C(i)}$.  
Put $J=\{G_i\}_{i\in \nn}$. 
Then $J$ is a sequence consisting of 
all finite subsets of $Q$ containing $q$. 

For each $i\in \nn$, 
let $a_i=(\alpha(G_i))^{-1}\cdot 2^{i^2}$. 
Put 
\[
X=\{q\}\sqcup\coprod_{i\in \nn}(G_i\setminus \{q\}), 
\]
and define a metric $d_X: X\times X\to [0, \infty)$ by 
\begin{align*}
&d_X(x, y)=\\
&
\begin{cases}
a_id_{G_i}(x, y) &\text{if $x, y\in F_i$ for some $i\in \nn$;}\\
a_id_{G_i}(x, q)+a_jd_{G_j}(q, y) &\text{if $x\in F_i$ and $y\in F_j$ for some $i\neq j$.}
\end{cases}
\end{align*}
The metric  space $X$ is a proper countable discrete metric space. 

We are going to prove  that every pointed proper length space is an asymptotic cone of $X$. 
To simplify our notation, 
for $R\in (0, \infty)$, 
and for $i\in \nn$, 
put $B_i(R)=(a_i)^{-1}B(q, a_iR; X)$. 
By the definition of $d_X$, 
the space $B_i(R)$ contains an isometric copy of 
$B(q, R; G_i)$ containing $p$. We denote by $S_i(R)$ that isometry copy. 
We also put $T_i(R)=B_i(R)\setminus S_i(R)$. 
Note that $S_i(R)\subset G_i$. 

\begin{lem}\label{lem:n+1}
Let $R\in (0, \infty)$. 
If $i\in \nn$ satisfies  $2^{i+1}\delta(G_i)>R$, 
then for every $k>i$ we have 
$B_i(R)\cap G_k=\emptyset$. 
\end{lem}
\begin{proof}
For every $x\in G_{k}$, 
by the definition of $d_X$, 
we have 
\[
d_X(q, x)\ge 2^{k^2}\ge 2^{(i+1)^2}.
\]
By Lemma \ref{lem:tukau}, 
we obtain 
\begin{align*}
2^{(i+1)^2}/a_i=2^{(i+1)^2-i^2}\alpha(G_i)\ge2^{2i+1}2^{-\pi_2(C(i))}\delta(G_i) \ge2^{i+1}\delta(G_i)>R. 
\end{align*}
Hence $a_iR<2^{(i+1)^2}$. 
This leads to the conclusion. 
\end{proof}

By Lemma \ref{lem:n+1} and by the definition of $T_i(R)$, 
we obtain: 
\begin{cor}\label{cor:T}
For every $i\ge 1$ and for every $R\in (0, \infty)$, 
we have 
\[
T_i(R)\subset \bigcup_{j=0}^{i-1}G_j. 
\]
\end{cor}

\begin{lem}\label{lem:16}
For every $i\ge 1$, 
we have 
\[
\alpha(G_{i})/\alpha(G_{i-1})< 16. 
\]
\end{lem}
\begin{proof}
By the conditions (B2), (C2) and  (C3), 
we obtain
\begin{align*}
\alpha(G_{i})/\alpha(G_{i-1})& < 2^{-\pi_2(C(i))+1+\pi_2(C(i-1))}\delta(G_i)/\delta(G_{i-1})\\
&\le 4\cdot \delta(G_i)/\delta(G_{i-1})\\
&< 4\cdot 2^{-\pi_3(C(i))+1+\pi_{3}(C(i-1))}\le 16. 
\end{align*}
This proves the lemma. 
\end{proof}

We conclude the following:
\begin{lem}\label{lem:32}
Let $i\in \nn$ and  $R\in (0, \infty)$. 
For all $x\in T_i(R)$, 
we have $(a_i)^{-1}d_X(q, x)< 32\cdot 2^{-i}$. 
In particular, 
we have 
\[
d_H(B_i(R), S_i(R); B_i(R))< 32\cdot 2^{-i}. 
\]
\end{lem}
\begin{proof}
By Corollary \ref{cor:T},  
we have 
$(a_i)^{-1}d_X(p, x)\le a_{i-1}\delta(G_{i-1})/a_i$. 
Lemmas \ref{lem:16} and \ref{lem:tukau} imply 
\begin{align*}
a_{i-1}\delta(G_{i-1})/a_{i} 
&=2^{(i-1)^2-i^2}\delta(G_{i-1})(\alpha(G_{i})/\alpha(G_{i-1}))\\
&<
16\cdot 2^{-2i+1}2^{-\pi_3(C(i-1))+1}\le 32\cdot 2^{-i}.
\end{align*}
This leads to the former part of the lemma. 
The later part follows from the former one,  
$q\in S_i(R)$ and $S_i(R)\subset B_i(R)$. 
\end{proof}

We now prove Theorem \ref{thm:asymcone}.
\begin{proof}[Proof of Theorem \ref{thm:asymcone}]
We first show that the metric  space 
$X$ constructed above is a desired space. 
By Proposition \ref{prop:homuni}, 
we may assume that $K\subset U$ and $p=q$. 
Since $K$ is proper,  
 we can take a sequence $\{A_i\}_{i\in \nn}$ of  finite subsets of $Q$ 
 satisfying the conditions (A1) and (A2) in Proposition \ref{prop:conv}. 
By the definition of $J=\{G_i\}_{i\in \nn}$, 
and by the condition (B3), there exists a strictly increasing map 
$\phi:\nn\to \nn$ such that $G_{\phi(i)}=A_i$ for every $i\in \nn$. 

To prove our statement, we next show that for each $R\in (0, \infty)$, 
the sequence  
$\{(B_{\phi(i)}(R), q)\}_{i\in \nn}$ converges to $(B(q, R; K), q)$ 
in the pointed Gromov--Hausdorff topology. 
Note that, 
by the conditions (C2) and (C3) and Lemma \ref{lem:tukau},  
we have $a_i\to \infty$ as $i\to \infty$. 

Since $\delta(G_{\phi(i)})\cdot 2^{\phi(i)+1} \to \infty$ as $i\to \infty$, 
we can take $N\in \nn$ such that for every $i\ge N$, 
we have  $R<\delta(G_{\phi(i)})\cdot 2^{\phi(i)+1}$. 
By Lemma \ref{lem:32},  
we have
\[
d_H(B_{\phi(i)}(R), S_{\phi(i)}(R); B_{\phi(i)}(R))< 32\cdot 2^{-\phi(i)}\le 32\cdot 2^{-i}.
\] 
Since $S_{\phi(i)}(R)$ is isometric to $B(q, R; A_i)$, 
by Proposition \ref{prop:conv}, 
 we conclude that $\{(B_{\phi(i)}(R), q)\}_{i\in \nn}$ converges to $(B(q, R; K), q)$ in the pointed Gromov--Hausdorff topology. 
Therefore $(K, p)$ is an asymptotic cone of $X$. 
Thus we conclude Theorem \ref{thm:asymcone}. 
\end{proof}

\begin{rmk}
Let $X$ be a metric space mentioned in Theorems \ref{thm:ury}, \ref{thm:tancone}, \ref{thm:asymcone} or Proposition \ref{prop:ury2}. 
By Theorem \ref{thm:cone} and Proposition \ref{prop:coneandp}, 
we obtain 
$\dim_AX=\infty$. 
\end{rmk}

All $(\omega_0+1)$-metric spaces  and all countable metic spaces have the topological dimension $0$,  and 
have the Hausdorff dimension $0$. 
Thus,  Theorems \ref{thm:ury},
\ref{thm:tancone}, \ref{thm:asymcone} or Proposition \ref{prop:ury2}
 tells us that analogies of Theorem \ref{thm:cone} 
 for the topological dimension, 
 the Hausdorff dimension and the conformal Hausdorff dimension are false. More precisely, we have the following:
\begin{prop}\label{prop:counter}
There exists a metric space $X$ such that for some $P\in \pc(X)$ we have
\[
\dim_{T}X<\dim_TP, \quad
\dim_HX<\dim_H P, \quad
 \cdim_HX<\cdim_HP, 
\] 
where $\dim_T$, 
$\dim_H$ and $\cdim_H$ stand for the topological dimension, 
the Hausdorff dimension and the conformal Hausdorff dimension,
 respectively. 
\end{prop}

\begin{rmk}
In \cite{CR}, Chen and Rossi studied a metric space 
containing a large class of metric spaces as tangent cones of it. 
They constructed a compact subset $X$ of $\rr^N$ 
 with $\dim_HX=0$ that  contains all similarity classes of  compact subsets of $[0, 1]^N$ as tangent cones at countable dense subset of $X$ (see \cite[Corollary 5.2]{CR}). 
 The metric space $X$ is an example failing  an analogy of Theorem \ref{thm:cone} for the Hausdorff dimension and the conformal Hausdorff dimension. 
\end{rmk}




\begin{thebibliography}{99}
\bibitem{A1}P. Assouad, \textit{Espace m\'etrique, plongements, facteurs}, 
U. E. R. Math\'ematique, Universit\'e Paris XI, Orsay, 1977. Th\`ese de doctorat, Publications Math\'ematique d'Orsay, No. 233-7769. 
\bibitem{A2}P. Assouad, \emph{\'Etude d'une dimension m\'etrique li\'ee \`a la possiblit\'e de plongements dans $\mathbb{R}^n$}, C. R. Acad. Sci. Paris S\'er. A-B, 288 (15) (1979),  A731--A734.
\bibitem{A3} P. Assouad, \textit{Prongements lipschitziens dans $\mathbb{R}^n$}, Bull. Soc. math. France, 111, (1983), 429--448.  
\bibitem{B}S. Banach, \textit{Th\'eorie des \'Operations Lin\'eaires}, 2nd ed., Chelsea Publishing Co., New York, 1963. 
\bibitem{BH}M. R. Breidson and A Haefliger, \textit{Metric Spaces of Non-Positive Curvature}, Grundlehren der Mathematischen Wissenschaften, 319. Springer, Berlin, 1999. 
\bibitem{BBI}D. Burago, Y. Burago and S. Ivanov, \textit{A Course in Metric Geometry}, Graduate Studies in Mathematics 33, Amer. Math. Soc. Providence, RI, 2001. 
\bibitem{CR}C. Chen and E. Rossi, \textit{Locally rich compact sets}, Illinois J.  Math.  58 (3) (2014), 779--806. 
\bibitem{DH}J. Dydak and J. Higes, \textit{Asymptotic cones and Assouad-Nagata dimension}, Proc. Amer. Math. Soc. 136 (6) (2008), 2225--2233. 
\bibitem{Fre}M. Fr\'echet, \textit{Les dimensions d'un ensemble abstrait}, Math. Ann. 68 (1910), 145--168.
\bibitem{FY}J. M. Fraser and H. Yu, \textit{Arithmetic patches, weak tangents and dimension}, Bull. London Math. Soc. 50 (2018), 85--95. 
\bibitem{Gro}M. Gromov, with appendices by M. Katz, P. Pansu, and S. Semmes, \textit{Metric Structures for Riemannian and Non-Riemannian Spaces}, (J. LaFontaine and P, Pansu, eds), Progress in Math. 152, Birkhauser, 1999.
\bibitem{H}J. Heinonen, \textit{Lectures on Analysis on Metric Spaces}, Springer-Verlag, New York, 2001.
\bibitem{H2}J. Heinonen, \textit{Geometrical embedding of metric spaces}, Report, University of Jyv\"askyl\"a Department of Mathematics and Statics, vol. 90. 2003. http://www.math.jyu.fi/tukimas/ber.html.
\bibitem{HW}W. Hurewicz and H. Wallman, \textit{Dimension Theory}, Revised ed., Princeton Univ. Press, 1948.
\bibitem{I} Y. Isihki, \textit{Quasi-symmetric invariant properties of Cantor metric spaces}, Ann.  Inst.  Fourier (Grenoble) 69 (6) (2019),  2681--2721. 
\bibitem{Ishi}Y. Ishiki, \textit{A characterization of metric subspace of  full Assouad dimension}, preprint, 2019: arXiv:1911.10775. 
\bibitem{KL}M. Kapovich and B. Leeb, \textit{Asymptotic cones and quasi-isometry classes of fundamental groups of $3$-manifolds}, Geom. Funct. Anal. 5 (3) (1995), 582--603. 
\bibitem{Lar}D. G. Larman, \textit{A new theory of dimension}, Proc. London. Math. Soc., 17 (1967), 178--192. 
\bibitem{LR}E. Le Donne and T. Rajala, \textit{Assouad dimension, Nagata dimension, and uniformly close metric tangents}, 
Indiana Univ. Math. J. 64 (1) (2015), 21--54. 
\bibitem{MT}J. M. Mackay and J. T. Tyson, \textit{Conformal Dimension Theory: Theory and Application}, Univ. Lecture Ser. 54, Amer. Math. Soc., 2010.
\bibitem{TV}P. Tukia and J. V\"ais\"al\"a, \textit{Quasisymmetric embeddings of metric spaces}, Ann. Acad. Sci. Fenn, Math. 5 (1980), 97--114. 
\bibitem{U}P. Urysohn, \textit{Sur un espace m\'etrique universel}, Bull. Sci. Math.,  51, 1927, 43--64 and 74--90.
\end{thebibliography}
\end{document}